\documentclass[final, 12pt]{amsart}
\newif\ifsharetheoremcounter\sharetheoremcountertrue

\usepackage{amsmath, amssymb, esint}
\usepackage{graphicx, xcolor}
\usepackage{microtype}

\ifdefined\ifenablehyperref\else\expandafter\newif\csname
        ifenablehyperref\endcsname\enablehyperreftrue\fi
\ifenablehyperref\usepackage[final]{hyperref}\else\fi

\ifdefined\ifdraft\else\expandafter\newif\csname ifdraft\endcsname\fi
\ifdim\overfullrule>0pt \drafttrue\else\draftfalse\fi

\ifdefined\ifsharetheoremcounter\else\expandafter\newif\csname
        ifsharetheoremcounter\endcsname\sharetheoremcounterfalse\fi

\allowdisplaybreaks[4]

\newcounter{modifiedcounter}

\ifdraft
    \newcommand{\fixme}[1]{\marginpar{\parbox{0in}{\fbox{\parbox{0.6255in}{
            \raggedright \scriptsize #1}}}}}
    
    \newcommand{\dummyline}[1]{\vspace{0.5em}\noindent\rule{\textwidth}{#1 pt}
            \vspace{0.5em}}

\else
    \newcommand{\fixme}[1]{}
    
    \newcommand{\dummyline}[1]{}

\fi

\newlength{\defaultarrayrulewidth}
\newcommand{\owedge}{\mathbin{\ooalign{$\bigcirc$\cr\hidewidth$\,\,\wedge$
        \hidewidth\cr}}}
\newcommand{\ddpfir}[1]{\frac{\partial}{\partial#1}}

\newcommand{\dddfir}[1]{\frac{d}{d#1}}
\newcommand{\dddnth}[2]{\frac{d^#2}{d#1^#2}}
\newcommand{\Boxg}[1]{\left(\ddpfir{t}-\Delta_{#1}\right)}

\theoremstyle{remark}
\numberwithin{equation}{section}

\newtheorem*{claim*}{Claim}

\newtheorem{theorem}{Theorem}[section]

\ifsharetheoremcounter
    \newtheorem{definition}[theorem]{Definition}
    \newtheorem{claim}[theorem]{Claim}
    \newtheorem{lemma}[theorem]{Lemma}
    \newtheorem{proposition}[theorem]{Proposition}
    \newtheorem{corollary}[theorem]{Corollary}
    \newtheorem{remark}[theorem]{Remark}
\else

    \newtheorem{lemma}{Lemma}[section]
    \newtheorem{proposition}{Proposition}[section]
    \newtheorem{corollary}{Corollary}[section]
    \newtheorem{remark}{Remark}[section]
\fi

\DeclareMathOperator{\Inj}{inj}

\DeclareMathOperator{\Ric}{Ric}

\DeclareMathOperator{\Rm}{Rm}

\DeclareMathOperator{\tr}{tr}
\DeclareMathOperator{\Vol}{Vol}
\DeclareMathOperator{\VolB}{VolB}

\newcommand{\Rick}[1]{\text{$#1$-Ricci}}

\DeclareMathOperator{\Area}{Area}
\DeclareMathOperator{\conj}{conj}

\begin{document}

\title[$3$-manifolds with $2$-Ricci curvature lower bound]{Volume comparison for $3$-manifolds with $2$-Ricci curvature lower bound and Ricci flow}

\author[S. Huang]{Shaochuang Huang}
\address[Shaochuang Huang]{School of Science, Shenzhen Campus of Sun Yat-sen University, No. 66, Gongchang Road, Guangming District, Shenzhen, Guangdong 518107, P. R. China.}
\email{huangshch23@mail.sysu.edu.cn}
\thanks{S. Huang is partially supported by a regular fund from Shenzhen Science
        and Technology Program No. JCYJ20240813151005007 and a start-up
        fund from SYSU}

\author[Z. Peng]{Zhuo Peng}
\address[Zhuo Peng]{School of Science, Shenzhen Campus of Sun Yat-sen University, No. 66, Gongchang Road, Guangming District, Shenzhen, Guangdong 518107, P. R. China.}
\email{pengzh55@mail2.sysu.edu.cn}


\begin{abstract}
    In this note, we study volume comparison for $3$-manifolds with $2$-Ricci curvature lower bound. We establish monotonicity and rigidity for geodesic sphere area and ball volume, and derive the pointed Gromov–Hausdorff precompactness under a uniform lower bound on the conjugate radius for $3$-manifolds with $2$-Ricci curvature lower bound. We also discuss the almost preservation of $2$-Ricci curvature lower bound and the  local non-collapsing propagation along Ricci flow. Finally, we  point out a topological rigidity result on nonnegative scalar curvature. 
\end{abstract}

\maketitle

\section{Introduction}

The study of topology and geometry of manifolds with sectional, Ricci, scalar curvature
        lower bound is one of the fundamental problems in differential geometry.
In this note, we study $3$-dimensional complete noncompact Riemannian manifolds
        with spectral $2$-Ricci curvature lower bound, namely the sum of the
        $2$ smallest eigenvalues of Ricci operator bounded from below, which is in between Ricci curvature and scalar curvature.  The notion of spectral  $k$-Ricci curvature appears first in Labbi's work \cite{Labbi} on the stability of positivity of spectral  $k$-Ricci curvature under surgeries, where the notion of this curvature condition is defined in an equivalent Grassmannian formulation. Later, Dussan and Noronha \cite{DN02PJM} study topological vanishing results for compact manifolds with $2$-nonnegative Ricci curvature by direct spectral definition. Wolfson defines general $k$-positive Ricci curvature directly by considering the sum of the
        $k$ smallest eigenvalues of Ricci operator is positive in \cite{Wolfson12LMSLNS}, where he studies connected sums and surgery preservation, together with consequences and conjectures concerning fundamental groups and fill radius for compact manifolds. In \cite{HKKZ}, Hirsch, Kazaras, Khuri and Zhang study width of Riemannian bands and related topics under $2$-Ricci curvature bounded from below by $4$. In \cite{CM-MathAnn-26}, Cucinotta and Mondino obtain a large-scale linear volume control under almost Ricci curvature lower bound and an integral positive lower bound in terms of  $2$-Ricci curvature. 

It is well-known that the scalar curvature controls the volume of balls at the infinitesimal level and the lower bound of scalar curvature alone cannot guarantee the Bishop-Gromov type volume comparison, see \cite{kkk-cvpde} and references therein for example. In this note, we study volume comparison for $3$-dimensional
        manifolds with $2$-Ricci curvature lower bound and we obtain the
        following monotonicity and rigidity for geodesic sphere area
        and ball volume.

\begin{theorem}
   Let $(M^3,g)$ be a $3$-dimensional Riemannian manifold with
            $\Rick{2}\ge-K$ for some $K\ge0$ and $p\in M$ with
            $\Inj_g(p)\geq R$ for some $R>0$. Then for any $0<r<R$, we have
    \[
        r\mapsto\frac{A_g(p,r)-A_{\bar{K}}(r)}{A_{\bar{K}}'(r)}
                \quad\text{and}\quad
                r\mapsto\frac{V_g(p,r)-V_{\bar{K}}(r)}{V_{\bar{K}}'(r)}
    \]
    are both non-increasing on $(0,R)$. In particular, we have
     \begin{equation} \label{eqn vol. upper}
        A_g(p,r)\le A_{\bar{K}}(r)\quad\text{and}\quad
                V_g(p,r)\le V_{\bar{K}}(r)
    \end{equation} for any $0<r<R$. 
    Moreover, any of the above inequalities becomes an equality for some
            $r_0\in(0,R)$ if and only if $B_g(p,r_0)$ is isometric to the
            geodesic ball of radius $r_0$ in the $3$-dimensional simply
            connected space form with constant sectional curvature $\bar{K}$. Here $A_{\bar{K}}(r)$ and $V_{\bar{K}}(r)$ are the area and volume of the
            comparison model space form with constant sectional curvature
            $\bar{K}:=-K/4$.
\end{theorem}

Furthermore, the area and volume upper bounds \eqref{eqn vol. upper} and rigidity can also be obtained if we relax the assumption on injectivity radius lower bound by conjugate radius lower bound, see Corollary \ref{cor vol. comp.}.

 As corollaries, we obtain volume doubling estimates and large-scale volume comparison. In particular, we establish the following Gromov-Hausdorff precompactness of
        $3$-dimensional pointed complete Riemannian manifolds with $2$-Ricci
        curvature lower bound and conjugate radius lower bound.

\begin{corollary}
    The set of $3$-dimensional pointed complete Riemannian manifolds
            with $\Rick{2}\ge-K$ for some fixed $K\ge0$ and
            conjugate radius uniformly bounded from below by $j_0>0$
            is precompact in the pointed Gromov-Hausdorff topology.
\end{corollary}

Note that the above pointed Gromov-Hausdorff limits allow collapsing. For example, the flat cylinder $\mathbb{R}^2\times\mathbb{R}/\varepsilon\mathbb{Z}$ equipped with the usual product metric has infinite conjugate radius with $\Rick{2}$ curvature vanishing for each $\varepsilon>0$   while their  pointed Gromov-Hausdorff limit is $\mathbb{R}^2$ as $\varepsilon\to0$.

In \cite{ST22JDG,ST21GT}, Simon and Topping study the local geometry of
        $3$-dimensional Riemannian manifolds with Ricci curvature lower bound
        and a coarse volume lower bound along Ricci flow.
In particular, they obtain the pseudo-locality result and short-time existence
        of Ricci flow in this case.
Furthermore, they prove the Anderson-Cheeger-Colding-Tian conjecture in this
        dimension, namely regularity of the Gromov-Hausdorff limits of complete
        pointed Riemannian manifolds in this case, via the Ricci flow they
        construct.
Meanwhile, they ask  whether or not Perelman's original pseudo-locality
        result can be extended in the analogous way to assume only some
        isoperimetric inequality rather than an almost-Euclidean one.
In this note, we will discuss the analogue of the above Simon-Topping's results for   $3$-dimensional Riemannian manifolds with
        spectral $2$-Ricci curvature lower bound and injectivity radius lower bound along Ricci flow.

Finally, we would like to point out the following interesting topological rigidity for manifolds with nonnegative scalar
        curvature, if one allows almost-Euclidean structure. The proof is just a simple observation from the short-time existence result of Ricci flow obtained by Cheng \cite{Cheng25AdvMath} and the diffeomorphism criterion via long-time Ricci flow by authors in this note \cite{HP26BLMS}, see also \cite{Wang20arXiv}.

\begin{corollary}
    For any $n\ge2$, there exists a constant $\delta(n)>0$ depending only on $n$
            such that every complete noncompact $n$-dimensional
            Riemannian manifold $(M^n,g)$ with nonnegative scalar curvature
            satisfying
    \[
        (\Area_g(\partial\Omega))^n
                \ge(1-\delta)n^n\omega_n(\Vol_g(\Omega))^{n-1}
    \]
    for any regular domain $\Omega\subset M$ is diffeomorphic to $\mathbb{R}^n$.
\end{corollary}

Throughout the paper, all manifolds are assumed to be smooth and connected, all Riemannian metrics are smooth, and all solutions to the Ricci flow are smooth solutions. All results in this paper were conceived and proved by the authors. ChatGPT-5.6-Sol was used to search the historical literature and to assist in checking all proofs for errors.

\section{Comparison geometry}\label{comap.}

In this section, we study Bishop-Gromov type volume comparison and pointed Gromov-Hausdorff
        precompactness for  $3$-manifolds with $2$-Ricci curvature
        lower bound.

Let $(M^n,g)$ be an $n$-dimensional Riemannian manifold and $\Lambda_1\le\Lambda_2\le\cdots
        \le\Lambda_n$ be the eigenvalues of Ricci operator.
For notational convenience, we denote by $\Rick{k}_g$ the sum of the
        $k$ smallest eigenvalues of Ricci operator, namely
        $\Rick{k}_g=\sum_{i=1}^k\Lambda_i$.
We may drop the subscript $g$ when there is no ambiguity.
For example, we write $\Rick{2}\ge0$ to mean $2$-nonnegative Ricci curvature.


We first point out that one can readily obtain the following volume comparison, which holds for any dimension $n\ge3$,  by further assuming an
        upper bound of scalar curvature, since this implies a lower bound of
        Ricci curvature.

\begin{proposition}
    For any $n\ge3$ and $A>0$, there exists a constant $C(n,A)>0$ with
            the following property.
    Let $(M^n,g)$ be a Riemannian manifold such that $\Rick{2}\ge-K_1$
            and $\mathcal{R}\le K_2$ for some $K_1,K_2\ge0$.
    Then for any $p\in M$ and $0<r\le A(K_2+(n-1)K_1)^{-1/2}$
            with $B_g(p,r)\subset\subset M$, we have
    \[
        \VolB_g(p,r)\le Cr^n.
    \]
\end{proposition}


\begin{proof}
    Let $\Lambda_1\le\Lambda_2\le\cdots\le\Lambda_n$ be the eigenvalues of
            $\Ric$ operator.
    Then $\Lambda_i\ge\Lambda_2\ge-(K_1+\Lambda_1),\;i=3,\ldots,n$.
    Therefore, we have
    \[
        K_2\ge\mathcal{R}
                =\Lambda_1+\sum_{i=2}^n\Lambda_i\ge-(n-2)\Lambda_1-(n-1)K_1,
    \]
    which implies $\Ric\ge-(n-1)\hat{K}g$, where
            $\hat{K}:=\frac{K_2+(n-1)K_1}{(n-1)(n-2)}$.
    By the Bishop-Gromov volume comparison, we have
    \begin{align*}
        \VolB_g(p,r)&\le \VolB_{-\hat{K}}(r)=n\omega_n\int_0^r\left(
                \frac{\sinh(\sqrt{\hat{K}}t)}{\sqrt{\hat{K}}}\right)^{n-1}dt \\
        &\le n\omega_n e^{(n-1)\sqrt{\hat{K}}r}\int_0^r t^{n-1}\,dt \\
        &\le \omega_n e^{\sqrt\frac{n-1}{n-2}A}r^n.
    \end{align*}
    This completes the proof.
\end{proof}


Next, we will show the main Bishop-Gromov type volume comparison in this note, namely monotonicity and rigidity for geodesic
        sphere area and ball volume for $3$-manifolds with $2$-Ricci curvature
        lower bound.

Since completeness is not assumed unless otherwise stated, we specify
    our conventions for the injectivity and conjugate radii.
For each $p\in M$, let $\Inj_g(p)$ denote the injectivity radius at $p$,
    defined as the supremum of all $\rho>0$ for which $\exp_p$ is
    defined on $B_\rho(0)\subset T_pM$ and is a diffeomorphism
    onto its image, see also \cite[pp. 165-166]{Lee18Book}.
Similarly, let $\conj_g(p)$ denote the conjugate radius at $p$,
    defined as the supremum of all $\rho>0$ for which $\exp_p$
    is defined and has nonsingular differential throughout $B_\rho(0)$.
Note that $\Inj_g(p)\le\conj_g(p)$ and that $B_g(p,R)\subset\subset M$
    whenever $0<R<\conj_g(p)$. As usual, we define \[\Inj_g(M):=\inf\limits_{p\in M}\Inj_g(p) \quad\text{and}\quad \conj_g(M):=\inf\limits_{p\in M}\conj_g(p).\]

Let $(M^3,g)$ be a $3$-dimensional Riemannian manifold with $\Rick{2}\ge-K$
        for some $K\ge0$.
For any $p\in M$, write
\[
    A_g(p,r):=\Area_g(\partial B_g(p,r))\quad\text{and}\quad
            V_g(p,r):=\VolB_g(p,r).
\]
    The area $\Area_g$ should be understood as the $2$-dimensional Hausdorff
            measure whenever $\partial B_g(p,r)$ is not smooth. The comparison model is the $3$-dimensional simply connected space form with
        constant sectional curvature $\bar{K}:=-K/4$.
Its sphere area and ball volume are given by
\begin{align*}
    &A_{\bar{K}}(r):=\begin{cases}
        \frac{8\pi}{K}\big(\cosh(\sqrt{K}r)-1\big), & K>0; \\
        4\pi r^2, & K=0,
    \end{cases} \\
    \text{and}\quad&V_{\bar{K}}(r):=\begin{cases}
        \frac{8\pi}{K^{3/2}}\big(\sinh(\sqrt{K}r)-\sqrt{K}r\big), & K>0; \\
        \frac{4}{3}\pi r^3, & K=0.
    \end{cases}
\end{align*}

\begin{theorem} \label{thm vol. comp.}
    Let $(M^3,g)$ be a $3$-dimensional Riemannian manifold with
            $\Rick{2}\ge-K$ for some $K\ge0$ and $p\in M$ with
            $\Inj_g(p)\ge R$ for some $R>0$.
    Then for any $0<r<R$, we have
    \[
        r\mapsto\frac{A_g(p,r)-A_{\bar{K}}(r)}{A_{\bar{K}}'(r)}
                \quad\text{and}\quad
                r\mapsto\frac{V_g(p,r)-V_{\bar{K}}(r)}{V_{\bar{K}}'(r)}
    \]
    are both non-increasing on $(0,R)$. In particular, we have
\[A_g(p,r)\le A_{\bar{K}}(r)\quad\text{and}\quad
                V_g(p,r)\le V_{\bar{K}}(r)
   \] for any $0<r<R$. 
    Moreover, either of the above inequalities becomes an equality for some
            $r_0\in(0,R)$ if and only if $B_g(p,r_0)$ is isometric to the
            geodesic ball of radius $r_0$ in the $3$-dimensional simply
            connected space form with constant sectional curvature $\bar{K}$.
\end{theorem}

    


\begin{proof}
The idea of the proof is to consider the second radial derivatives of area of geodesic sphere and then to apply Gauss equation and the Gauss-Bonnet
            theorem. This argument goes back to Schoen and Yau's work \cite{SY79AnnMath} on the study of compact stable minimal surface in $3$-dimensional compact manifolds with positive scalar curvature, see also \cite{GL-IHES-83}.
    For every $0<r<R$, the sphere $\partial B_g(p,r)$ is a smooth embedded
            surface in $M$.
    Let $h$ denote the second fundamental form of $\partial B_g(p,r)$ with
            respect to the outward unit normal vector field $\partial_r$
            and $H:=\tr h$ be the mean curvature.
    For any $x\in\partial B_g(p,r)$, choose an orthonormal basis
            $\{e_1,e_2,e_3\}$ of $T_xM$ such that $e_3=\partial_r$.
    Denote by $K_x$ the Gauss curvature of $\partial B_g(p,r)$ at $x$.
    Then by the Gauss equation, we have
    \begin{equation}\begin{aligned} \label{eqn Gauss K}
        K_x&=R_{1212}+\det h \\
        &=\frac{1}{2}(R_{11}+R_{22}-R_{33})+\frac{1}{2}(H^2-|h|^2).
    \end{aligned}\end{equation}
    Direct computation (see also \cite{Li12Book, Gimeno21BLMS})
            shows that
    \begin{equation}\begin{aligned} \label{eqn ddr2 A}
        A_g''(p,r)&=\dddnth{r}{2}\int_{\partial B_g(p,r)}d\sigma
                =\dddfir{r}\int_{\partial B_g(p,r)}H\,d\sigma \\
        &=\int_{\partial B_g(p,r)}\partial_rH+H^2\,d\sigma \\
        &=\int_{\partial B_g(p,r)}-R_{33}-|h|^2+H^2\,d\sigma \\
        &=\int_{\partial B_g(p,r)}2K_x-(R_{11}+R_{22})\,d\sigma \\
        &\le2\int_{\partial B_g(p,r)}K_x\,d\sigma+K A_g(p,r) \\
        &=8\pi+K A_g(p,r),
    \end{aligned}\end{equation}
    where we have used $\partial_rH=-\Ric(\partial_r,\partial_r)-|h|^2$ (see
            \cite[(1.137)]{CLN06Book}), \eqref{eqn Gauss K} and the Gauss-Bonnet
            theorem.
    Clearly, it holds that
    \[
        A_{\bar{K}}''(r)=8\pi+K A_{\bar{K}}(r)\quad\text{and hence}\quad
                A_{\bar{K}}'''(r)=K A_{\bar{K}}'(r).
    \]

    Now we are ready to prove the monotonicity of area comparison.
Write
    \begin{align*}
        &D_A(r):=A_{\bar{K}}(r)-A_g(p,r), \\
        \text{and}\quad&W_A(r):=D_A'(r)A_{\bar{K}}'(r)-D_A(r)A_{\bar{K}}''(r).
    \end{align*}
    Since $A_{\bar{K}}'(r)>0$, we have
    \begin{align*}
        W_A'(r)&=D_A''(r)A_{\bar{K}}'(r)-D_A(r)A_{\bar{K}}'''(r) \\
        &=A_{\bar{K}}'(r)(D_A''(r)-K D_A(r))\ge0.
    \end{align*}
    Recalling that $\lim\limits_{r\to0^+}D_A(r)=0$ and $\lim\limits_{r\to0^+}D'_A(r)=0$, we have  $\lim\limits_{r\to0^+}W_A(r)=0$
            and hence $W_A(r)\ge0$.
    Therefore, we obtain
    \[
        \dddfir{r}\left(\frac{A_{\bar{K}}(r)-A_g(p,r)}{A_{\bar{K}}'(r)}\right)
                =\frac{W_A(r)}{A_{\bar{K}}'(r)^2}\ge0
    \]
    for all $r\in(0,R)$. Recalling that $\lim\limits_{r\to0^+}\frac{D_A(r)}{A'_{\bar{K}}(r)}=0$, we have $A_g(p,r)\le A_{\bar{K}}(r)$.
    This completes the proof of the area comparison.

    Since $V_g'(p,r)=A_g(p,r)$ and $V_{\bar{K}}'(r)=A_{\bar{K}}(r)$, we have
    \begin{align*}
        &\dddfir{r}\left(\frac{V_{\bar{K}}(r)-V_g(p,r)}{V_{\bar{K}}'(r)}\right)
                =\frac{D_A(r)A_{\bar{K}}(r)-(V_{\bar{K}}(r)-V_g(p,r))
                A_{\bar{K}}'(r)}{A_{\bar{K}}(r)^2} \\
        &\qquad=\frac{A_{\bar{K}}'(r)}{A_{\bar{K}}(r)^2}\left(
                \frac{D_A(r)}{A_{\bar{K}}'(r)}\int_0^rA_{\bar{K}}'(s)\,ds
                -\int_0^rD_A(s)\,ds\right) \\
        &\qquad=\frac{A_{\bar{K}}'(r)}{A_{\bar{K}}(r)^2}\int_0^rA_{\bar{K}}'(s)
                \left(\frac{D_A(r)}{A_{\bar{K}}'(r)}
                -\frac{D_A(s)}{A_{\bar{K}}'(s)}\right)ds\ge0,
    \end{align*}
    where we have used the monotonicity of area comparison
            in the last inequality. $V_g(p,r)\le V_{\bar{K}}(r)$ follows from the fact that $\lim\limits_{r\to0^+}\frac{V_{\bar{K}}(r)-V_g(p,r)}{A_{\bar{K}}(r)}=0$.

    It remains to show the rigidity of the comparison. The implication from isometry to equality is straightforward and we only
            show the converse.
    If $A_g(p,r_0)=A_{\bar{K}}(r_0)$ or $V_g(p,r_0)=V_{\bar{K}}(r_0)$ for some
            $r_0\in(0,R)$, then by the monotonicity results, we have
            $A_g(p,r)=A_{\bar{K}}(r)$ for all $r\in(0,r_0]$.
    So the inequality in \eqref{eqn ddr2 A} becomes an equality for all
            $r\in(0,r_0]$, which means that $R_{11}+R_{22}=-K$ and
            hence $\Ric(\partial_r,\partial_r)\ge-K/2$ for all $r\in(0,r_0]$.
    Then
    \begin{equation} \label{eqn ddr H}
        \partial_rH=-\Ric(\partial_r,\partial_r)-|h|^2
                \le-\frac{H^2}{2}+\frac{K}{2}.
    \end{equation}
    A similar argument as in the proof of \cite[Lemma 1.126]{CLN06Book} shows
            that $H\le H_{\bar{K}}$ for all $r\in(0,r_0]$, where $H_{\bar{K}}$
            is the mean curvature of the geodesic sphere of radius $r$ in the
            $3$-dimensional simply connected space form with constant sectional
            curvature $\bar{K}$.
    Then
    \begin{align*}
        0&\le\int_{\partial B_g(p,r)}H_{\bar{K}}-H\,d\sigma \\
        &=A_g(p,r)H_{\bar{K}}(r)-A_g'(p,r) \\
        &=A_{\bar{K}}(r)H_{\bar{K}}(r)-A_{\bar{K}}'(r)=0.
    \end{align*}
    So $H\equiv H_{\bar{K}}$ and equality holds in \eqref{eqn ddr H}.
    It follows that $\Ric(\partial_r,\partial_r)=-K/2$, which forces
            $\Ric=-(K/2)g$ on $B_g(p,r_0)$.
    Note that the Riemann curvature tensor is determined by the Ricci tensor
            in dimension $3$.
    This completes the proof.
\end{proof}


\begin{corollary} \label{cor vol. comp.}
    Let $(M^3,g)$ be a $3$-dimensional Riemannian manifold with
            $\Rick{2}\ge-K$ for some $K\ge0$ and $p\in M$ with
            $\conj_g(p)\ge R$ for some $R>0$.
    Then for any $0<r<R$, we have
    \[
        A_g(p,r)\le A_{\bar{K}}(r)\quad\text{and}\quad
                V_g(p,r)\le V_{\bar{K}}(r).
    \]
    Moreover, the rigidity statement of Theorem \ref{thm vol. comp.}
            remains valid.
\end{corollary}


\begin{proof}
For any fixed $r\in(0,R)$, we can choose $\bar{R}\in(r,R)$. Consider the pullback metric $\tilde{g}:=\exp_p^*g$ on
            $B_{\bar{R}}(0)\subset T_pM$.
    Then we have $\Inj_{\tilde{g}}(0)\ge\bar{R}$.
    The curvature hypothesis passes to $\tilde{g}$, so Theorem
            \ref{thm vol. comp.} applies. For each $0<\rho\le r$, let
    \[
        C_\rho:=\{v\in\partial B_\rho(0)\mid d_g(p,\exp_p(v))=\rho\}.
    \]
    Since $\exp_p$ is $1$-Lipschitz  in the sense that 
    \[
        d_g(\exp_p(v),\exp_p(w))\le d_{\tilde{g}}(v,w)\quad\text{for all}\quad
                v,w\in B_{\bar{R}}(0),
    \]
    and maps $C_\rho$ onto $\partial B_g(p,\rho)$, by
            \cite[Chapter 2, \S1, (1.8)]{Simon18Notes}, we have
    \[
        A_g(p,\rho)\le \Area_{\tilde{g}}(C_\rho)\le A_{\tilde{g}}(0,\rho)
                \le A_{\bar{K}}(\rho).
    \] Integrating the area comparison yields the volume comparison.

    Suppose equality holds in either comparison at some $r_0\in(0,R)$.
    We claim that $\Inj_g(p)\ge r_0$.
    Indeed, for any $v\in B_{r_0}(0)$, define
    \[
        F(v):=|v|-d_g(p,\exp_p(v)).
    \]
    If $\Inj_g(p)<r_0$, then there exists $v_0\in B_{r_0}(0)$ such that
            $F(v_0)>0$.
    By continuity of $F$ and the star-shapedness of the minimizing domain,
            the nonminimizing vectors contain a nonempty open set,
            which forces both comparisons to be strict at $r_0$,
            which is a contradiction. Since $r_0<\conj_g(p)$, applying Theorem \ref{thm vol. comp.} to the
            pullback metric $\tilde{g}$ yields that $B_{\tilde{g}}(0,r_0)$
            is isometric to the model geodesic ball.
    Then the isometry of $B_g(p,r_0)$ follows from $\Inj_g(p)\ge r_0$.

    Conversely, suppose $B_g(p,r_0)$ is isometric to the model geodesic ball.
    Clearly, $V_g(p,r_0)=V_{\bar{K}}(r_0)$ and $\Inj_g(p)\ge r_0$.
    However, the injectivity on the open ball does not imply the injectivity on
            its boundary.
    Suppose $\exp_p$ maps two distinct points $q_1,q_2\in\partial B_{r_0}(0)$
            to $q\in\partial B_g(p,r_0)$.
    Choose disjoint neighborhoods of $q_1,q_2$ in $\partial B_{r_0}(0)$ and denote
            their image boundary sheets by $\Sigma_1$ and $\Sigma_2$, with
            outward unit normals $\nu_1$ and $\nu_2$.
    Then we must have $\nu_1(q)=-\nu_2(q)$ and hence every point in
            $\partial B_g(p,r_0)$ has at most two preimages on
            $\partial B_{r_0}(0)$.
    Since the outward second fundamental forms of $\Sigma_1$ and $\Sigma_2$ are positive
            definite, these two sheets meet only at $q$ locally.
    The compactness of $\partial B_g(p,r_0)$ and the local injectivity of
            $\exp_p$ imply that we can find a neighborhood of
            $q$ in $\partial B_g(p,r_0)$ in which there is
            no other point with two preimages.
    Therefore, there are only finitely many points in $\partial B_g(p,r_0)$
            having two preimages, which ensures that
            $A_g(p,r_0)=A_{\bar{K}}(r_0)$.
    This completes the proof.
\end{proof}


As a corollary, we first obtain the following volume doubling estimates.

\begin{corollary} \label{thm vol. doubling}
    Let $(M^3,g)$ be a $3$-dimensional  Riemannian manifold
            with $\Rick{2}\ge-K$ for some $K\ge0$ and $p\in M$.
    Suppose that $\inf_{x\in B_g(p,r)}\Inj_g(x)>2r$.
    Then there exists a constant $C>0$ depending only on an upper bound for
            $\sqrt{K}r$ such that
    \[
        V_g(p,2r)\le C V_g(p,r).
    \]
\end{corollary}


\begin{proof}
    We only show the case $K>0$ here, since the case $K=0$ is similar and simpler.
    By Theorem \ref{thm vol. comp.}, we have
    \[
        V_g(p,2r)\le V_{\bar{K}}(2r)-\frac{A_{\bar{K}}(2r)}{A_{\bar{K}}(r)}
                (V_{\bar{K}}(r)-V_g(p,r)).
    \]
    Since
    \[
        \frac{A_{\bar{K}}(2r)}{A_{\bar{K}}(r)}
                =\frac{\cosh(2\sqrt{K}r)-1}{\cosh(\sqrt{K}r)-1}
                =2\cosh(\sqrt{K}r)+2,
    \]
    and 
    \[
        V_g(p,r)\ge\frac{64}{27\pi}r^3\] by Croke's inequality \cite[Proposition 14]{Croke80AnnSci},  we have
    \begin{align*}
        \frac{V_g(p,2r)}{V_g(p,r)}&\le\frac{A_{\bar{K}}(2r)}{A_{\bar{K}}(r)}
                +\left(V_{\bar{K}}(2r)-\frac{A_{\bar{K}}(2r)}{A_{\bar{K}}(r)}
                V_{\bar{K}}(r)\right)\bigg/V_g(p,r) \\
        &\le2\cosh(\sqrt{K}r)+2+\frac{27\pi^2}{4}\cdot
                \frac{\sqrt{K}r\cosh(\sqrt{K}r)-\sinh(\sqrt{K}r)}{(\sqrt{K}r)^3} \\
        &=:f(\sqrt{K}r).
    \end{align*}
    Note that $f$ is increasing on $[0,\infty)$ with
            $\lim\limits_{x\to0^+}f(x)=4+9\pi^2/4$.
    This completes the proof.
\end{proof}


If we further assume the conjugate radius bound is uniform over the whole
        manifold, we can obtain the following large-scale volume comparison,
        which is a key to show the pointed Gromov-Hausdorff precompactness in
        this setting. We would like to first point out that a global positive  conjugate radius lower bound has already implied the completeness of the manifold, see \cite[Problem 6-12, pp. 188]{Lee18Book}.

\begin{proposition} \label{prop vol. growth}
    Let $(M^3,g)$ be a $3$-dimensional  Riemannian manifold with
            $\Rick{2}\ge-K$ for some $K\ge0$.
    Suppose $\conj_g(M)\ge j_0>0$.
    Let $L:=\min\{j_0,1/\sqrt{K}\}$.
    Then there exists a universal constant $C>0$ such that
            for any $p\in M$ and $R>0$,
    \[
        V_g(p,R)\le CL^3e^{\frac{C}{L}R}.
    \]
\end{proposition}


\begin{proof}
    Let $\rho:=\frac{1}{40(\sqrt{K}+j_0^{-1})}\in[\frac{L}{80},\frac{L}{40}]$.
    Then $\rho\le\frac{j_0}{40}$ and $\sqrt{K}\rho\le\frac{1}{40}$. Let $X\subset M$ be a maximal $\rho$-separated set
            (see \cite[Exercise 1.6.4]{BBI01Book}) and $Y:=X\cap B_g(p,R+\rho)$.
    By the definition of maximal $\rho$-separated set, we know the balls
            $\{B_g(x,\rho/2)\}_{x\in X}$ are disjoint and the balls
            $\{B_g(x,\rho)\}_{x\in X}$ cover $M$.
    For any $x\in B_g(p,R)$, there exists some $y\in X$ such that
            $d_g(x,y)<\rho$.
    Then $d_g(p,y)<R+\rho$ and hence $y\in Y$.
    So $B_g(p,R)\subset\bigcup_{y\in Y}B_g(y,\rho)$.
Since $\sqrt{K}\rho\le\frac{1}{40}$, for all $r\in(0,\frac{7}{2}\rho]$,
            there exists a universal constant $C_1>1$ such that
    \begin{equation} \label{eqn V barK}
        V_{\bar{K}}(r)\le\omega_3r^3e^{7\sqrt{K}\rho/2}\le C_1r^3.
    \end{equation}
Then by Corollary \ref{cor vol. comp.}, we have
    \[
        V_g(y,\rho)\le V_{\bar{K}}(\rho)\le C_1\rho^3.
    \]

    Denote by $|Y|$ the cardinality of $Y$.
    Then our goal is to estimate $|Y|$.
    Fix $x_0\in Y$ with $d_g(p,x_0)<\rho$.
    For any $y\in Y$, take a minimizing geodesic from $p$ to $y$ and choose
            consecutive subdivision points $p_0=p,p_1,\ldots,p_m=y$ such that
            $d_g(p_{i-1},p_i)\le\rho$ for $i=1,\ldots,m$.
    For every $p_i,\,i=1,\ldots,m-1$, we can find a $x_i\in X$ such that
            $d_g(p_i,x_i)<\rho$.
    Thus we find a chain of points $x_0,x_1,\ldots,x_m=y$ in $X$ with
    \[
        d_g(x_{i-1},x_i)\le d_g(x_{i-1},p_{i-1})+d_g(p_{i-1},p_i)+d_g(p_i,x_i)
                <3\rho,
    \]
    for $i=1,\ldots,m$.
    Since $d_g(p,y)<R+\rho$, we may choose $m:=\lceil(R+\rho)/\rho\rceil
            <R/\rho+2$.
    We claim that there exists some universal constant $C_2>0$ such that
            for all $x\in M$, the cardinality of the set $U_x:=\{z\in X\mid
            d_g(x,z)<3\rho\}$ is bounded by $C_2$.
    Then we have $|Y|\le C_2^{R/\rho+2}$ and hence
   \[
        V_g(p,R)\le C_2^{R/\rho+2}C_1\rho^3\le
                C_1C_2^2\rho^3e^{(\log C_2)R/\rho}.
    \]

    It remains to show the claim.
    Take a complete Riemannian metric $\tilde{g}$ on $T_xM$ agreeing with
            $\exp_x^*g$ on $B_{j_0/2}(0)$ and with the Euclidean metric outside
            $B_{3j_0/4}(0)$.
    Then $\Inj_{\tilde{g}}(0)\ge j_0/2$.
    For any $v\in B_{j_0/8}(0)$, by \cite[Theorem 1.2]{Xu18CommunContempMath},
            we have
    \begin{align*}
        \Inj_{\tilde{g}}(v)&\ge\min\{\Inj_{\tilde{g}}(0),\conj_{\tilde{g}}(v)\}
                -d_{\tilde{g}}(0,v) \\
        &\ge\left(\frac{j_0}{2}-\frac{j_0}{8}\right)-\frac{j_0}{8}
                \ge\frac{j_0}{4}.
    \end{align*}
    For each $z\in U_x$, choose a lift $v_z\in B_{3\rho}(0)$.
    These lifts are $\rho$-separated with respect to $\tilde{g}$.
    Then
    \[
        \bigsqcup_{z\in U_x}B_{\tilde{g}}(v_z,\frac{\rho}{2})
                \subset B_{7\rho/2}(0)\subset B_{j_0/8}(0).
    \]
    Applying Theorem \ref{thm vol. comp.} to the pullback metric $\exp^\ast_g$ on $B_{j_0/2}(0)$, and using \eqref{eqn V barK} and Croke's inequality
            \cite[Proposition 14]{Croke80AnnSci}, we obtain
    \[
        |U_x|\cdot\frac{64}{27\pi}(\frac{\rho}{2})^3
                \le V_{\tilde{g}}(0,\frac{7\rho}{2})
                \le V_{\bar{K}}(\frac{7\rho}{2})
                \le C_1\left(\frac{7\rho}{2}\right)^3.
    \]
    This completes the proof by taking $C_2:=\frac{27\pi}{8}C_1(\frac{7}{2})^3$.
\end{proof}

\begin{corollary} \label{cor pGH compact.}
    The set of $3$-dimensional pointed complete Riemannian manifolds
            with $\Rick{2}\ge-K$ for some fixed $K\ge0$ and
            conjugate radius uniformly bounded from below by $j_0>0$
            is precompact in the pointed Gromov-Hausdorff topology.
\end{corollary}

\begin{proof}
    It suffices to verify the conditions \cite[Theorem 8.1.10]{BBI01Book}.
Fix any $R,\varepsilon>0$, take $\rho$ as in the proof of Proposition
            \ref{prop vol. growth} and set $\delta:=\min\{\varepsilon,\rho\}$.
    For any $(M,g,p)$ in the provided set, choose a maximal $\delta$-separated
            subset $Y\subset \overline{B_g(p,R)}$.
    An argument similar to  the proof of Proposition \ref{prop vol. growth}
            shows that $|Y|$ is uniformly bounded from above by a constant
            depending only on $K,j_0,R$ and $\varepsilon$.
    This completes the proof.
\end{proof}


\section{Some discussions related to Ricci
        flow}\label{RF}

As we mention in the introduction, Simon and Topping \cite{ST22JDG,ST21GT} study the local geometry of
        $3$-dimensional Riemannian manifolds with Ricci curvature lower bound
        and a coarse volume lower bound along Ricci flow. They obtain pseudo-locality result and short-time existence of Ricci flow in the above setting. Using their Ricci flow, they establish the regularity of the Gromov-Hausdorff limits of complete
        pointed Riemannian manifolds with uniform Ricci curvature lower bound
        and uniform  volume lower bound of unit balls. The mechanism of their proof contains two key ingredients. One is  the almost preservation of lower bound of Ricci
        curvature, another one is  the non-collapsing propagation along $3$-dimensional
        Ricci flow in a scale larger than curvature-scale. 

In this section, we will first study the preservation of lower bound of $2$-Ricci curvature along $3$-dimensional Ricci flow. We first point out that the preservation of $2$-nonnegative Ricci curvature along $3$-dimensional Ricci flow is well-known by experts.  In \cite[Proposition 6.2]{CX25JMPA}, Cao and Xie show that the $2$-nonnegative
        Ricci curvature is preserved under complete $3$-dimensional
        Ricci flow with bounded curvature.
In fact, one can remove the assumption on bounded curvature just by taking
        $k=2$ and $K_k=0$ in Chen's result
        \cite[Corollary 2.3 (ii)]{Chen09JDG}


\begin{proposition}
    Let $(M^3,g(t)),\,t\in[0,T]$ be a smooth complete solution to the Ricci flow.
    If $\Rick{2}_{g(0)}\ge0$, then $\Rick{2}_{g(t)}\ge0$
            for all $t\in[0,T]$.
\end{proposition}

\dummyline{0.5}

\begin{proof}
    This is exactly the special case of \cite[Corollary 2.3 (ii)]{Chen09JDG}
            with $k=2$ and $K_k=0$, since
            $\Rick{2}_{g(0)}=2m_1+m_2+m_3$.
    Here $m_1\le m_2\le m_3$ are the eigenvalues of the curvature operator.
\end{proof}


In order to obtain the pseudo-locality result or construct local Ricci flow by the techniques developed in
        \cite{ST22JDG,ST21GT}, we will investigate the local preservation of lower
        bound of $2$-Ricci curvature.

\begin{lemma} \label{lma pt. alge.}
    Let $(M^3,g(t))_{t\in[0,T)}$ be a $3$-dimensional Ricci flow and
            $\Lambda_1\le\Lambda_2\le\Lambda_3$ be the eigenvalues of the
            Ricci operator.
    Then $\Lambda_1+\Lambda_2$ is continuous on $M\times[0,T)$ and satisfies
    \[
        \Boxg{g(t)}(\Lambda_1+\Lambda_2)\ge\frac{7}{8}(\Lambda_1+\Lambda_2)^2
    \]
    on $M\times(0,T)$ in the barrier sense.
\end{lemma}


\begin{proof}
    Let $m_1\le m_2\le m_3$ be the eigenvalues of the curvature operator.
    Then $\Lambda_1=m_1+m_2$, $\Lambda_2=m_1+m_3$
            and $\mathcal{R}=2(m_1+m_2+m_3)$.
 Hence
    \begin{align*}
        &\Boxg{g(t)}(\Lambda_1+\Lambda_2)=\Box(\frac{\mathcal{R}}{2}+m_1) \\
        &\qquad\ge2(2m_1^2+m_2^2+m_3^2+m_1m_2+m_1m_3+2m_2m_3) \\
        &\qquad=2[(2m_1+m_2+m_3)^2-3m_1(2m_1+m_2+m_3)+4m_1^2] \\
        &\qquad\ge\frac{7}{8}(2m_1+m_2+m_3)^2 \\
        &\qquad=\frac{7}{8}(\Lambda_1+\Lambda_2)^2,
    \end{align*}
    in the barrier sense.
    Here we denote $\Boxg{}$ by $\Box$ for notational convenience.
    This completes the proof.
\end{proof}

Due to Lemma \ref{lma pt. alge.}, we know that  $2$-Ricci curvature lower bound behaves like scalar curvature lower bound. The following proposition is a general form of \cite[Lemma 8.1]{ST22JDG} and the proof here adapts the cutoff argument there. See also \cite{HT18AJM,LT22CJM,Simon13GT}.

\begin{proposition} \label{lemma st22 lma8.1}
    For any $n\ge2,\;c_0,c_1,K>0$ and $\gamma\in(0,1)$, there exist
            $\hat{T}(n,c_0,\gamma)>0$ and $\sigma(c_1,K,\gamma)>0$
            with the following property.
    Suppose that $(M^n,g(t))$ is a Ricci flow for $t\in[0,T)$, and $x_0\in M$
            satisfies $B_{g(t)}(x_0,1)\subset\subset M$ for all $t\in[0,T)$.
    Let $f$ be a continuous function on $M\times[0,T)$ satisfying
    \[
        \Boxg{g(t)}f\ge c_1f^2
    \]
    on $M\times(0, T)$ in the barrier sense.
    We assume further that
    \begin{enumerate}
        \item[\textup{(\romannumeral1)}]
            $f(\cdot,0)\ge-K$ on $B_{g(0)}(x_0,1)$;
        \item[\textup{(\romannumeral2)}]
            $\Ric_{g(t)}\le\frac{c_0(n-1)}{t}$ on $B_{g(t)}(x_0,\sqrt{t})$
                    for all $t\in(0,T)$.
    \end{enumerate}
    Then for all $t\in[0,T\wedge\hat{T})$, we have
    \[
        f(\cdot,t)\ge-Ke^{\sigma t}\quad\text{on}\quad B_{g(t)}(x_0,1-\gamma).
    \]
\end{proposition}

\begin{proof}[Proof of Proposition \ref{lemma st22 lma8.1}]
    As in the proof of \cite[Lemma 8.1]{ST22JDG}, let $\hat{T}(n,c_0,\gamma),
            \allowbreak k(\gamma),\allowbreak V(\gamma)>0$ be the constants
            asserted in \cite[Lemma 7.1]{ST22JDG} with
            $\varepsilon=\frac{1}{4}$, $r_1=1-\gamma$ and $r_2=1$,
            and $h$ be the corresponding cut-off function.
    For any fixed $\delta>0$ and $(x,t)\in M\times[0,\hat{T}\wedge T)$, let
    \[
        \ell(x,t):=h(x,t)f(x,t)+(1+\delta)Ke^{bt},
    \]
    where $b>0$ is a constant to be determined later.
    Clearly we have $\ell(\cdot,0)>0$ throughout $M$, and $\ell(x,t)>0$ for all
            $x\in M\setminus B_{g(t)}(x_0,1)$ and $t\in[0,\hat{T}\wedge T)$.

    We claim that $\ell(x,t)>0$ for all $(x,t)\in M\times[0,\hat{T}\wedge T)$.
    Otherwise, let $(x_1,t_1)\in M\times(0,\hat{T}\wedge T)$ be the first point
            where $\ell(x_1,t_1)=0$.
    Then $h(x_1,t_1)f(x_1,t_1)<0$ and hence $h(x_1,t_1)>0$ and $f(x_1,t_1)<0$.
    By \cite[Lemma 7.1 \textup{(\romannumeral2)}]{ST22JDG}, the assumption
            on $f$, and the continuity of $h$ and $f$, there exists a spacetime
            neighborhood $\mathcal{U}$ of $(x_1,t_1)$ on which smooth barrier
            functions $H$ and $F$ are defined.
    More precisely, $H$ is a lower barrier for $h$ with $h\ge H>0$, while $F$
            is an upper barrier for $f$ with $f\le F<0$.
    Both agree with the corresponding functions at $(x_1,t_1)$ and satisfy
            the respective differential inequalities.
    Define
    \[
        \mathcal{L}(x,t):=H(x,t)F(x,t)+(1+\delta)Ke^{bt},
    \]
    within $\mathcal{U}$.
    Then $\mathcal{L}\ge\ell\ge0$ on $\mathcal{U}\cap\{t\le t_1\}$ with
            $\mathcal{L}(x_1,t_1)=\ell(x_1,t_1)=0$.
    Hence, at $(x_1,t_1)$, we have
    \begin{equation} \label{eqn nabla L}
        0=\nabla\mathcal{L}=F\nabla H+H\nabla F,
    \end{equation}
    and
    \begin{align*}
        0&\ge\Boxg{g(t)}\mathcal{L} \\
        &=F\Box H+H\Box F-2\langle\nabla H,\nabla F\rangle+b(1+\delta)Ke^{bt} \\
        &\ge c_1H F^2+2V^2H^\frac{1}{2}F+bK \\
        &\ge bK-\frac{V^4}{c_1},
    \end{align*}
    where we have used \eqref{eqn nabla L} and \cite[Lemma 7.1
            \textup{(\romannumeral2)}]{ST22JDG} in the second inequality.
    Choosing $b(c_1,K,\gamma)=\frac{V^4}{c_1K}+1$, we obtain a contradiction.

    Therefore, for all $t\in[0,\hat{T}\wedge T)$ and $x\in
            B_{g(t)}(x_0,1-\gamma)$, we have
    \[
        e^{-kt}f(x,t)+(1+\delta)Ke^{bt}>0.
    \]
    Letting $\delta\to0$, we conclude that
    \[
        f(x,t)\ge-Ke^{(b+k)t}.
    \]
    This completes the proof.
\end{proof}


\begin{corollary}\label{2-Ricci-preservation}
    For any $c_0,K>0$ and $\gamma\in(0,1)$, there exist
            $\hat{T}(c_0,K,\gamma)>0$ and $\sigma(K,\gamma)>0$
            with the following property.
    Suppose that $(M^3,g(t))$ is a $3$-dimensional Ricci flow for $t\in[0,T)$,
            and $x_0\in M$ satisfies $B_{g(t)}(x_0,1)\subset\subset M$ for all
            $t\in[0,T)$.
    We assume further that
    \begin{enumerate}
        \item[\textup{(\romannumeral1)}]
            $\Rick{2}_{g(0)}\ge-K$ on $B_{g(0)}(x_0,1)$;
        \item[\textup{(\romannumeral2)}]
            $\Ric_{g(t)}\le\frac{2c_0}{t}$ on $B_{g(t)}(x_0,\sqrt{t})$
                    for all $t\in(0,T)$.
    \end{enumerate}
    Then for all $t\in[0,T\wedge\hat{T})$, we have
    \[
        \Rick{2}_{g(t)}\ge-Ke^{\sigma t}\ge-2K
                \quad\text{on}\quad B_{g(t)}(x_0,1-\gamma).
    \]
\end{corollary}


\begin{proof}
    By Proposition \ref{lemma st22 lma8.1} and Lemma \ref{lma pt. alge.},
            there exist $\hat{T}(c_0,\gamma)>0$ and $\sigma(K,\gamma)>0$
            such that for all $t\in[0,T\wedge\hat{T})$, we have
            $\Rick{2}_{g(t)}\ge-Ke^{\sigma t}$.
    Shrinking $\hat{T}>0$ to ensure that $e^{\sigma\hat{T}}\le2$,
            this completes the proof.
\end{proof}

\begin{remark}
    It is interesting to compare Corollary \ref{2-Ricci-preservation} with \cite[Proposition 2.2]{Chen09JDG} by Chen, where Chen assumes a fixed-scale upper Ricci bound, while we assume  a curvature-scale upper Ricci bound on shrinking balls.  
\end{remark}


Once we obtain local preservation of $2$-Ricci curvature lower bound, next step is to investigate the local non-collapsing propagation along $3$-dimensional Ricci flow. By applying Croke's \cite{Croke80AnnSci} and Wang's
        \cite[Theorem 5.7]{Wang18CambJ} results, we obtain the
        no-local-collapsing theorem up to  curvature-scale.

\begin{theorem}\label{non-collapsing}
    For any $n\ge3,\,\alpha,i_0,R>0$ and $\Lambda\ge0$, there exist
            $\bar{T}(n,\alpha,i_0,R,\Lambda)>0$,
            $\rho(n,\alpha)\in(0,1/\sqrt{\alpha})$
            and $\kappa(n)>0$ with the following property.
    Let $(M^n,g(t)),\,t\in[0,T]$ be a complete Ricci flow with uniformly
            \emph{bounded curvature}.
    Suppose for some $p\in M$,
    \[
        \mathcal{R}_{g(0)}\ge-\Lambda,\quad\Inj_{g(0)}\ge i_0
                \quad\text{on}\quad B_{g(0)}(p,R)
    \]
    and
    \[
        |\Rm|_{g(t)}\le\frac{\alpha}{t}\quad\text{on}\quad
                B_{g(t)}(p,\sqrt{t})\quad\text{for}\quad t\in(0,T].
    \]
    Then for any $t\in(0,\bar{T}\wedge T]$ and $r\in(0,\rho\sqrt{t}]$, we have
    \[
        \VolB_{g(t)}(p,r)\ge\kappa r^n.
    \]
\end{theorem}


\begin{proof}
    Choose $A:=\max\{1000n,\sqrt{n}\alpha/(n-1)\}$ and
            $\rho:=\min\{1,1/\sqrt{n\alpha}\}$.
    Then for any $t\in(0,T]$ and $r\in(0,\rho\sqrt{t}]$, we have
            for all $s\in(0,t]$,
    \[
        \Ric_{g(s)}\le\sqrt{n}|\Rm|_{g(s)}
                \le\frac{\sqrt{n}\alpha}{s}\le\frac{(n-1)A}{s}
                \quad\text{on}\quad B_{g(s)}(p,\sqrt{s}),
    \]
    and
    \[
        \mathcal{R}_{g(t)}\le n|\Rm|_{g(t)}
                \le\frac{n\alpha}{t}\le\frac{1}{\rho^2t}\le r^{-2}
                \quad\text{on}\quad B_{g(t)}(p,r)\subset B_{g(t)}(p,8A\sqrt{t}).
    \]
    Let $\Omega_t := B_{g(0)}(p,20A\sqrt{t})$.
    By \cite[Theorem 5.7]{Wang18CambJ}, we have
    \[
        \frac{\VolB_{g(t)}(p,r)}{\omega_n r^n}\ge\exp\Big\{
                -2^{n+7}-2+\inf_{\tau\in[t,2t]}\mu(\Omega_t,g(0),\tau)\Big\}.
    \]

    It suffices to find a $\bar{T}>0$ such that
            $\inf_{\tau\in[t,2t]}\mu(\Omega_t,g(0),\tau)$ is bounded from below
            by a constant depending only on $n$ for all
            $t\in(0,\bar{T}\wedge T]$.
    Choose $\bar{T}>0$ sufficiently small such that
    \[
        20A\sqrt{\bar{T}}<R,\quad40A\sqrt{\bar{T}}<i_0\quad\text{and}\quad
                \Lambda\bar{T}\le1.
    \]
    Then $\Omega_t\subset B_{g(0)}(p,R)$ and
            $\mathrm{diam}_{g(0)}(\Omega_t)\le40A\sqrt{t}<i_0$.
    By \cite[Theorem 11]{Croke80AnnSci}, there exists a constant $C_n>0$
            depending only on $n$ such that
    \[
        \inf_{D\subset\subset\Omega_t}\frac{|\partial D|}{|D|^\frac{n-1}{n}}
                \ge C_n.
    \]
    Here area and volume are measured with respect to $g(0)$.
    Therefore, \cite[Lemma 3.5 (3.24)]{Wang18CambJ} implies that
            for any $\tau\in[t,2t]$,
    \[
        \mu(\Omega_t,g(0),\tau)\ge\log\frac{C_n^n}{\omega_nn^n}-\Lambda\tau
                \ge\log\frac{C_n^n}{\omega_nn^n}-2.
    \]
    This completes the proof.
\end{proof}

However, the above local non-collapsing propagation is still not enough to
        obtain the pseudo-locality type result in our setting.
In particular, we still cannot obtain the analogue of Lemma 2.1 in
        \cite{ST22JDG} under the assumptions of $2$-Ricci curvature lower bound and uniform injectivity radius lower bound.
Because the proof of Lemma 2.1 in \cite{ST22JDG} is a contradiction blow-up
        argument, which requires large-scale non-collapsing and Theorem
        \ref{non-collapsing} gives only non-collapsing at fixed-scale after
        blow-up.
Note that the above local non-collapsing propagation only assumes scalar
        curvature lower bound and holds in any dimension $n\geq 3$. It is still
        interesting to improve this non-collapsing propagation under
        $2$-Ricci curvature lower bound in $3$-dimensional Ricci flow.

Finally, we would like to  point out the following topological rigidity results if one allows almost-Euclidean structure. The proof is just a simple observation by applying Cheng's short-time existence of Ricci flow \cite{Cheng25AdvMath} and our work \cite{HP26BLMS} on diffeomorphism criterion via Ricci flow, see also \cite{Wang20arXiv}. To our knowledge, Theorem \ref{thm entropy diff.} and Corollary \ref{thm isop. diff.} below have not previously appeared in this generality. Theorem 1.5 and Theorem 4.5 in \cite{Cheng25AdvMath} assume exact Euclidean isoperimetric inequality or non-negativity of $\nu$-entropy to conclude isometry to  Euclidean space, while  our results give a natural diffeomorphism gap version of Cheng’s metric rigidity theorems. For the definition of local and global $\nu$-entropy, one may see \cite{Wang18CambJ}.

\begin{theorem} \label{thm entropy diff.}
    For any $n\ge2$, there exists a constant $\varepsilon(n)>0$ depending only
            on $n$ such that every complete noncompact $n$-dimensional
            Riemannian manifold $(M^n,g)$ satisfying $\nu(M,g)\ge-\varepsilon$
            is diffeomorphic to $\mathbb{R}^n$.
\end{theorem}


\begin{proof}
    Take $\varepsilon:=\delta^2$, where $\delta>0$ is the constant in
            \cite[Theorem 4.3]{Cheng25AdvMath} with $\alpha=\frac{1}{4(n-1)}$.
    For any positive $T_i\to+\infty$, applying
            \cite[Theorem 4.3]{Cheng25AdvMath} with $\alpha=\frac{1}{4(n-1)}$,
            $T=T_i$ and $\eta=\frac{1}{2}$, we obtain a sequence of
            complete Ricci flows $g_i(t)$ on $M\times[0,C_nT_i]$ with
            $g_i(0)=g$ satisfying
    \[
        |\Rm|(x,t)\le\frac{\alpha}{t}
    \]
    and
    \[
        \inf_{\rho\in(0,\alpha^{-1}\sqrt{t})}\frac{\VolB_{{g_i(t)}}(x,\rho)}{\rho^n}
                \ge(1-\alpha)\omega_n
    \]
    for $(x,t)\in M\times(0,C_nT_i]$, where $C_n=(\epsilon/2)^2$ and
            $\epsilon$ is the constant asserted in
            \cite[Theorem 4.3]{Cheng25AdvMath}.
    By similar arguments as in the proof of \cite[Theorem 1.3]{CHL24AnnPDE},
            we can extract convergent subsequence in locally smooth sense
            to obtain a complete long-time solution $g(t)$ to the Ricci flow with
            $g(0)=g$ satisfying the same curvature estimates and volume
            estimates for $(x,t)\in M\times(0,+\infty)$.
    For any $t>0$, consider $\tilde{g}:=\frac{4\alpha^2}{t}g(t)$.
    By Cheeger-Gromov-Taylor's injectivity radius estimate
            \cite[Theorem 4.7]{CGT82JDG} and rescaling back,
            we have
    \[
        \Inj_{g(t)}\ge i_0\sqrt{t}\quad\text{for some constant}\quad i_0(n)>0.
    \]
    The result then follows from \cite[Theorem 1.1]{HP26BLMS}.
\end{proof}


\begin{corollary} \label{thm isop. diff.}
    For any $n\ge2$, there exists a constant $\delta(n)>0$ depending only on $n$
            such that every complete noncompact $n$-dimensional
            Riemannian manifold $(M^n,g)$ with nonnegative scalar curvature
            satisfying
    \[
        (\Area_g(\partial\Omega))^n
                \ge(1-\delta)n^n\omega_n(\Vol_g(\Omega))^{n-1}
    \]
    for any regular domain $\Omega\subset M$ is diffeomorphic to $\mathbb{R}^n$.
\end{corollary}


\begin{proof}
    Take $\delta:=1-e^{-\varepsilon}$, where $\varepsilon>0$ is the constant
            in Theorem \ref{thm entropy diff.}.
    By \cite[Lemma 3.5 (3.25)]{Wang18CambJ}, we have
    \[
        \nu(\Omega,g,\tau)\ge n\log(1-\delta)^{1/n}=-\varepsilon
    \]
    for any regular domain $\Omega\subset M$ and $\tau>0$, which implies
    \[
        \nu(M,g)\ge-\varepsilon.
    \]
    The result then follows from Theorem \ref{thm entropy diff.}.
\end{proof}


\appendix

\section{Algebraic curvature estimates}\label{Algebraic}

In this appendix, we would like to show the following algebraic curvature estimates for
        manifolds with $\Rick{2}\ge0$ in any dimension $n\ge3$, which may be of independent interest.

\begin{proposition}
    Let $(M^n,g),\;n\ge3$ be a Riemannian manifold with $\Rick{2}\ge0$.
    Then $\mathcal{R}\ge0$ and the following assertions hold:
    \begin{enumerate}
        \item[\textup{(\romannumeral1)}] $\displaystyle
                \Ric\ge-\frac{\mathcal{R}}{n-2}g$;
        \item[\textup{(\romannumeral2)}] $|\Ric|\le c_n\mathcal{R}$,
                where $\displaystyle c_n=\left\{\;\begin{aligned}
                    &\sqrt{3},\quad n=3,\\
                    &1,\quad n\ge4;
                \end{aligned}\right.$
        \item[\textup{(\romannumeral3)}] $|\Rm|\le C_n\mathcal{R}$,
                if we further assume $g$ has nonnegative isotropic curvature
                (Weakly PIC) when $n\ge4$, where $C_n>0$ is a constant
                depending only on $n$.
            Here $|\Rm|$ denotes the standard norm of the $(0,4)$-tensor.
            In particular, one can take $C_3=\sqrt{11}$ and $C_4=\sqrt{3}$.
    \end{enumerate}
\end{proposition}


\begin{proof}
    Let $\Lambda_1\le\Lambda_2\le\cdots\le\Lambda_n$ be the eigenvalues of
            $\Ric$ operator.
    Then $\Lambda_i\ge\Lambda_2\ge\max\{-\Lambda_1,0\},\;i=3,\ldots,n$.
    Therefore, we have $\mathcal{R}\ge0$ and
    \[
        \mathcal{R}=\Lambda_1+\sum_{i=2}^n\Lambda_i\ge-(n-2)\Lambda_1.
    \]
    Hence $\Lambda_1\ge-\frac{\mathcal{R}}{n-2}$, which implies
            \textup{(\romannumeral1)}.

    By the definition, we have
    \[
        \mathcal{R}^2-|\Ric|^2
                =(\sum_{i=1}^n\Lambda_i)^2-\sum_{i=1}^n\Lambda_i^2
                =2\sum_{i<j}\Lambda_i\Lambda_j.
    \]
    If $\Lambda_1\ge0$, then $\sum_{i<j}\Lambda_i\Lambda_j\ge0$ and hence
            $|\Ric|\le\mathcal{R}$.
    If $\Lambda_1<0$, then
    \begin{align*}
        \sum_{i<j}\Lambda_i\Lambda_j&=\sum_{i=2}^n\Lambda_1\Lambda_i
                +\sum_{2\le i<j\le n}\Lambda_i\Lambda_j \\
        &=\Lambda_1\sum_{i=2}^n(\Lambda_i+\Lambda_1)-(n-1)\Lambda_1^2
                +\sum_{2\le i<j\le n}(\Lambda_i+\Lambda_1)(\Lambda_j+\Lambda_1) \\
        &\qquad+\frac{1}{2}(n-1)(n-2)\Lambda_1^2-\Lambda_1
                \sum_{2\le i<j\le n}(\Lambda_i+\Lambda_1+\Lambda_j+\Lambda_1) \\
        &=\frac{1}{2}(n-1)(n-4)\Lambda_1^2
                -(n-3)\Lambda_1\sum_{i=2}^n(\Lambda_i+\Lambda_1) \\
        &\qquad+\sum_{2\le i<j\le n}(\Lambda_i+\Lambda_1)(\Lambda_j+\Lambda_1) \\
        &\ge0,
    \end{align*}
    provided $n\ge4$.
    For $n=3$, we have $0\le\Lambda_2\le\Lambda_3
            \le\mathcal{R}=\Lambda_1+\Lambda_2+\Lambda_3$ and
            $\Lambda_1\ge-\Lambda_2\ge-\mathcal{R}$.
    Then $|\Lambda_i|\le\mathcal{R},\;i=1,2,3$ and hence
            $|\Ric|\le\sqrt{3}\mathcal{R}$.
    This proves \textup{(\romannumeral2)}.

    It remains to show \textup{(\romannumeral3)}.
    For $n=3$, the Weyl tensor vanishes and by \cite[(1.57)]{CLN06Book},
            we have
    \[
        \Rm=-\frac{\mathcal{R}}{4}g\owedge g+\Ric\owedge g.
    \]
    For any fixed point $p\in M$, we can choose an orthonormal basis
            $\{e_1,e_2,e_3\}$ of $T_pM$ such that
            $\Ric(e_i,e_j)=\Lambda_i\delta_{ij},\;i,j=1,2,3$.
    Then
    \[
        R_{ijij}=-\frac{\mathcal{R}}{2}+\Lambda_i+\Lambda_j
                \quad\text{for}\quad1\le i<j\le3.
    \]
    It follows that
    \begin{align*}
        |\Rm|^2&=\sum_{i,j,k,l=1}^3R_{ijkl}^2=4\sum_{1\le i<j\le3}R_{ijij}^2 \\
        &=4\sum_{k=1}^3\left(\frac{\mathcal{R}}{2}-\Lambda_k\right)^2 \\
        &=4|\Ric|^2-\mathcal{R}^2\le11\mathcal{R}^2.
    \end{align*}

    For $n\ge4$, we argue by contradiction.
    Suppose there exist $\Rm_k$ for $k=1,2,\ldots$ with $2$-nonnegative Ricci
            curvature and nonnegative isotropic curvature
            such that $|\Rm_k|>k\mathcal{R}_k$.
    Let
    \[
        \widetilde{\Rm}_k:=\frac{\Rm_k}{|\Rm_k|}.
    \]
    Then $|\widetilde{\Rm}_k|=1$ and $\widetilde{\mathcal{R}}_k<1/k$.
    Since the set of unit-norm algebraic curvature operators with
            $2$-nonnegative Ricci curvature and nonnegative isotropic
            curvature is compact, after passing to a subsequence, there exists
            a limit $\widetilde{\Rm}_\infty$ such that
            $|\widetilde{\Rm}_\infty|=1$ and $\widetilde{\mathcal{R}}_\infty=0$.
    It follows that $\widetilde{\Ric}_\infty=0$ and hence by
            \cite[Proposition 7.3]{Brendle10Book}
            (see also \cite[Proposition 2.5]{MW93Duke}), we have
            $\widetilde{\Rm}_\infty=0$, which contradicts
            $|\widetilde{\Rm}_\infty|=1$.
    In particular, for $n=4$, we can write
    \[
        \Rm=\begin{pmatrix}
            A & B \\
            B^T & C
        \end{pmatrix}
    \]
    with respect to the decomposition $\wedge^2=\wedge^+\oplus\wedge^-$ after we choose a local orientation. Denote by $A_1\le A_2\le A_3$ the eigenvalues of $A$ and
            $B_3\ge0$ the largest singular value of $B$. Since $\tr(A)=\frac{\mathcal{R}}{4}$ and $A_1+A_2\ge0$, we have
            $A_3\le\frac{\mathcal{R}}{4}$ and hence $-\frac{\mathcal{R}}{4}\le
            A_i\le\frac{\mathcal{R}}{4},\;i=1,2,3$.
    Then $|A|^2\le3\left(\frac{\mathcal{R}}{4}\right)^2
            =\frac{3}{16}\mathcal{R}^2$.
    Similarly, we have $|C|^2\le\frac{3}{16}\mathcal{R}^2$.
    By \cite[Lemma 2.2]{CX25JMPA}, we know that $2$-nonnegative Ricci curvature
            implies $B_3\le\frac{\mathcal{R}}{4}$ and hence
            $|B|^2\le3\left(\frac{\mathcal{R}}{4}\right)^2
            =\frac{3}{16}\mathcal{R}^2$.
    Therefore, we obtain
    \[
        |\Rm|^2=4\big(|A|^2+2|B|^2+|C|^2\big)\le3\mathcal{R}^2.
    \]
    This completes the proof of \textup{(\romannumeral3)}.
\end{proof}

\bibliographystyle{amsplain}
\bibliography{88-refs}

\end{document}